\documentclass[12pt]{amsart}

\usepackage{amssymb,latexsym}

\usepackage{enumerate}

\makeatletter

\@namedef{subjclassname@2010}{%

  \textup{2010} Mathematics Subject Classification}

\makeatother

\newtheorem{thm}{ \bf Theorem}[section]
\newtheorem{cor}[thm]{ \bf Corollary}
\newtheorem{lem}[thm]{ \bf Lemma}

\newtheorem{rem}[thm]{ \bf Remark}

\newcommand{\be}{\begin{equation}}
\newcommand{\ee}{\end{equation}}
\newcommand{\bea}{\begin{eqnarray}}
\newcommand{\eea}{\end{eqnarray}}
\newcommand{\Bea}{\begin{eqnarray*}}
\newcommand{\Eea}{\end{eqnarray*}}

\numberwithin{equation}{thm}

\def\CF{{\mathcal F}}
\def\CH{{\mathcal H}}
\def\CA{{\mathcal A}}
\def\CB{{\mathcal B}}
\def\CS{{\mathcal S}}
\def\CM{{\mathcal M}}
\def\C{{\mathbb C}}
\def\H{{\mathbb H}}
\def\N{{\mathbb N}}
\def\R{{\mathbb R}}

\begin{document}

\title[Analytic vectors]
{On the structure of analytic vectors for the Schr\"odinger representation}
\author{R. Garg and S. Thangavelu}

\address{Department of Mathematics\\ Indian Institute of Science\\Bangalore-560 012}
\email{rahulgarg@math.iisc.ernet.in, veluma@math.iisc.ernet.in}

%---------------------------------------------------------------------------

\keywords{Heisenberg group, Schr\"odinger representations, Hermite and Laguerre functions, Weyl transform, Hermite semigroup, Poisson-Hermite semigroup}

\subjclass[2010]{Primary 42C15 ; Secondary 42B35, 42C10, 42A56.}

%---------------------------------------------------------------------------

\begin{abstract}
This article deals with the structure of analytic and entire vectors for the 
Schr\"{o}dinger representations of the Heisenberg group. Using refined 
versions of Hardy's theorem and their connection with Hermite expansions we 
obtain very precise representation theorems for analytic and entire vectors.
\end{abstract}

\maketitle

\section{Introduction}
\setcounter{equation}{0}

In a recent work Gimperlein-Kr\"otz-Lienau \cite{GKL} have proved a factorisation theorem of Dixmier-Malliavin type for the space of analytic vectors associated to a representation of a real Lie group $G.$ To be more specific, let $ (\pi,V) $ be a representation of moderate growth of a real Lie group $ G $ on a Fr\'echet space $ V.$ The authors have produced a natural algebra $ A(G) $ such that $ \pi(A(G))V = V^\omega $ where $ V^\omega $ stands for the space of analytic vectors for $ \pi.$ As a consequence, they have shown that the space $ V^\omega $ coincides with the space of analytic vectors for the Laplace-Beltrami operator on $ G.$ This extends a result of Dixmier-Malliavin \cite{DM} where the space $ V^\infty $ of smooth vectors was considered.

The study of analytic vectors has a long history going back to the work of Nelson \cite{N} in 1959. In the context of nilpotent Lie groups, R. Goodman \cite{G1}, \cite{G2} and R. Penney \cite{P} have studied the spaces of analytic and entire vectors for unitary representations of nilpotent Lie groups. Thus, the work in \cite{GKL} is a far reaching generalisation even in the context of nilpotent Lie groups.

In the present article our goal is modest: we restrict ourselves to the Schr\"odinger representations $ \pi_\lambda $ of the Heisenberg group $ \H^n.$ The space of analytic vectors for $ \pi_\lambda $ has been well studied (by Goodman, Penney and others) and there is a close connection between Hermite expansions and analytic vectors. Our aim is to exploit this connection and prove very precise statements about the structures of analytic and entire vectors. We state results for Schr\"odinger representations $ \pi_\lambda$ for any $\lambda \neq 0$, but for simplicity the proofs are written for $\lambda = 1$ only.

The paper is organised as follows: In section 2 we recall some basic definitions and prove a representation theorem (Theorem 2.2) for the smooth vectors.

We will consider the analytic vectors in section 3. The motivation comes from the work of Hille (\cite{Hi}), where he gave a necessary and sufficient condition for a vector to be analytic. Exploring the asymptotic properties of Hermite functions, he proved it in 1-dimension. We prove the higher dimensional analogue of the result (Theorem 3.3), and then the following representation theorem for analytic vectors.

For a function $ f(x,u,t) $ on $ \H^n $ we define $ f^\lambda $ to be the function on $\R^{2n}$ obtained by taking the inverse Fourier transform of $ f $ at $ \lambda $ in the last variable. We let $ \CF_1 f^\lambda $ and $ \CF_2 f^\lambda$ to stand for the Fourier transforms of $ f^\lambda $ in the $ x $ and $ u $ variables respectively. Let $\CB^\lambda$ be the space of all functions $f $ in $ L^1(\H^n)$ which satisfy the estimates
$$ |\CF_1 f^\lambda(x,u)| \leq C e^{-t(\frac{1}{|\lambda|}x^2 + \frac{|\lambda|}{4} u^2)^{\frac{1}{2}}},~~~ |\CF_2 f^\lambda(x,u)| \leq C e^{-t(\frac{|\lambda|}{4} x^2 +
\frac{1}{|\lambda|}u^2)^{\frac{1}{2}}} $$
for some $ t > 0.$

\begin{thm} The space $ \CB^\lambda $ defined above is an algebra under convolution and we have $ \pi_\lambda(\CB^\lambda)V = V^\omega $ where $ V = L^2(\R^n).$
\end{thm}

As expected, it turns out that the study of analytic (resp. entire) vectors for the Schr\"odinger representation $ \pi_\lambda $ is intimately related to the Poisson-Hermite semigroup $ e^{-t\sqrt{H(\lambda)}}$ (resp. the Hermite semigroup $ e^{-t H(\lambda)}$). Moreover, it is also connected with Hardy's uncertainty principle as we will see in Section 4.

For each $t > 0$ the image of $ L^2(\R^n) $ under the Hermite semigroup $e^{-tH(\lambda)}$ is a subclass of entire vectors for $\pi_\lambda $ and is called the Hermite-Bergman space denoted by $ \CH_t^\lambda(\C^n).$ For $\lambda = 1$, we write $\CH_t(\C^n)$ for $\CH_t^1(\C^n).$ Let $ E_t^\lambda $ consists of functions on $\R^n$ which are restrictions of functions from $\CH_t^\lambda(\C^n)$. Using some ramifications of Hardy's theorem we obtain the following representation theorem for the union and intersection of $\CH_t^\lambda(\C^n).$ For each $t > 0$ and $\lambda \neq 0$ we let $\CA_t^\lambda$ stand for the space of all functions $ f $ in $ L^1(\H^n) $ satisfying
$$ |\CF_1 f^\lambda(x,u)| \leq C e^{-\tanh(t\lambda) (\frac{1}{\lambda}x^2 + \frac{\lambda}{4}u^2)},~~~~ |\CF_2 f^\lambda(x,u)| \leq C e^{-\tanh(t\lambda) (\frac{\lambda}{4}x^2 + \frac{1}{\lambda}u^2)}.$$
Then both $ \cup_{t>0}\CA_t^\lambda $ and $\cap_{t>0} \CA_t^\lambda $ are algebras under convolution.

\begin{thm} For every $ \lambda \neq 0 $ we have
$$ \pi_\lambda(\cup_{t>0}\CA_t^\lambda)V = \cup_{t>0} E_t^\lambda,~~ \pi_\lambda(\cap_{t>0}\CA_t^\lambda)V = \cap_{t>0}E_t^\lambda. $$
\end{thm}

We then deduce Theorem 4.7 from the one dimensional result of Vemuri \cite{V} concerning the decay of Hermite coefficients of functions satisfying Hardy's condition. There is a generalisation of Hardy's theorem in 1-dimension, due to Pfannschmidt \cite{Pf}. Using Theorem 4.7, we obtain its exact analogue in higher dimension.

\begin{thm} Let $f$ be a measurable function on $ \R^n $ satisfying
$$ |f(x)| \leq P(x) e^{-\frac{1}{2}a x^2},~~ |\hat{f}(\xi)| \leq Q(\xi) e^{-\frac{1}{2}b\xi^2} $$
where $a, b > 0 $ and $ P, Q $ are such that
$$ \limsup_{|x|\rightarrow \infty}\frac{\log{P(x)}}{|x|^2}~=~0~= ~\limsup_{|y|\rightarrow \infty}\frac{\log{Q(y)}}{|y|^2}.$$
Then,\\
(i) if $ ab > 1,$ then $ f \equiv 0, $ \\
(ii) if $ab = 1 $, then $f$ and $\hat{f}$ are entire functions on $\C^n $ of the form
$$ f(z)= R(z)e^{-\frac{1}{2}az^2}, ~~~ \hat{f}(z) = S(z)e^{-\frac{1}{2a}z^2},$$
where $R(z)$ and $S(z)$ are entire functions on $\C^n $ of order 2 and at most of minimal type, and \\
(iii) there are infinitely many functions (e.g. all the Hermite functions) satisfying the above estimates for $ ab < 1 $ .
\end{thm}

As an immediate consequence of Theorem 1.3 and Theorem 4.2 we obtain the following characterisation of $\CH_\infty(\C^n) = \cap_{t>0} \CH_t^{\lambda}(\C^n) = \cap_{t>0} \CH_t(\C^n).$

\begin{thm} An entire function $ F $ on $\C^n$ belongs to $ \CH_\infty(\C^n) $ if and only if its restriction $f$ to $\R^n$ satisfies the Pfannschmidt condition (the condition stated in Theorem 1.3) with $a = b = 1$. Consequently, there exists $ G \in \CH_\infty(\C^n) $ such that $ G = \hat{f}$ on $ \R^n.$ Moreover, both $ F(z) e^{\frac{1}{2}z^2} $ and $G(z)e^{\frac{1}{2}z^2}$ are entire functions on $\C^n$ of order 2 and at most of minimal type.
\end{thm}

%------------------------------------------------------------------------

\section{Preliminaries and smooth vectors}

We look at the Heisenberg group $ \H^n = \R^n \times \R^n \times \R $ with group law
$$(x,u,t)(y,v,s) = (x+y,u+v, t+s+\frac{1}{2}(u \cdot y - v \cdot x)).$$ For this group we have a family of irreducible unitary representations $ \pi_\lambda $ parameterized by non-zero reals $ \lambda $ called the Schr\"odinger representations. All of them are realised on the same Hilbert space, namely $ L^2(\R^n) $ and given explicitly by
$$ \pi_\lambda(x,u,t)\varphi(\xi) = e^{i \lambda t} e^{i \lambda (x \cdot \xi + \frac{1}{2}x \cdot u)} \varphi(\xi+u) $$ where $ \varphi \in L^2(\R^n).$ When $t=0$, we write $\pi_\lambda(x,u)$ for $\pi_\lambda(x,u,0)$. The integrated representations on $ L^1(\H^n) $ are also denoted by the same symbol, namely $ \pi_\lambda .$ Given an integrable function $ f $ on $ \H^n $ we see that $ \pi_\lambda(f) = W_\lambda(f^\lambda) $ where
$$ f^\lambda(x,u) = \int_\R f(x,u,t) e^{i\lambda t} dt$$
is the inverse Fourier transform of $ f $ at $ \lambda $ in the last variable and
$$ W_\lambda(f^\lambda) = \int_{\R^{2n}} f^\lambda(x,u) \pi_\lambda(x,u) dx du $$
is known as the Weyl transform of $ f^\lambda.$ For $f, g \in L^1(\H^n)$ the convolution of f with g is defined by
$$(f \ast g)(x,u,t) = \int_{\H^n} f((x,u,t)(-y,-v,-s))g(y,v,s)dy dv ds.$$
It follows that
$$ (f \ast g)^\lambda = f^\lambda \ast_\lambda g^\lambda $$
where $\ast_\lambda$ is called the $\lambda$-twisted convolution and is defined for functions $F,G$ on $\R^{2n}$ by
$$ (F \ast_\lambda G)(x,u) = \int_{\R^{2n}} F(x-y,u-v)G(y,v) e^{i\frac{\lambda}{2} (u \cdot y - v \cdot x)} dy dv.$$
We then have
$$ W_\lambda((f \ast g)^\lambda)) = W_\lambda(f^\lambda \ast_\lambda g^\lambda) = W_\lambda(f^\lambda)W_\lambda(g^\lambda).$$
For $\lambda = 1$ we write $ \pi $ and $ W $ instead of $ \pi_1 $ and $ W_1$, and denote the twisted convolution by $\times.$

It is well known that the space of smooth vectors for $ \pi $ is precisely the space $ \CS(\R^n) $ of Schwartz functions. The following result which can be considered as a representation theorem for smooth vectors was proved in \cite{GZ}.

\begin{thm}[Gr\"ochenig-Zimmermann]A function $ \psi $ is in $\CS(\R^n) $ if and only if it can be written as $ \psi = W(F)\varphi$ for some $\varphi \in \CS(\R^n)$ where $ F $ satisfies
$$ |F(x,u)| \leq C_m (1+x^2+u^2)^{-m} $$
for all $ m \in \N.$
\end{thm}

Let $ e^{itH} $ and $ e^{itL} $ be the unitary groups generated by the Hermite and special Hermite operators $ H = -\Delta+|x|^2 $ and $ L = -\Delta+\frac{1}{4}|z|^2 - i\sum_{j=1}^n (x_j\frac{\partial}{\partial y_j} - y_j\frac{\partial}{\partial x_j}) $ respectively. Then it can be easily verified that $ W(e^{itL}F) = W(F)e^{itH}.$ Moreover, it is well known that $ e^{-i\frac{\pi} {4}H } = e^{-i \frac{n \pi}{4}} \CF $ and $e^{-i\frac{\pi}{2}L} = e^{-i \frac{n \pi}{2}} \CF_s$ where $\CF$ is the Fourier transform on $ \R^n $ and $ \CF_s $ is the symplectic Fourier transform on $ \R^{2n}.$ In view of these facts, it follows that $ W(F)\varphi $ is a Schwartz function if either $ F $ or its symplectic Fourier transform satisfies the decay conditions in the above theorem. We would also like to strengthen the above result by taking $ \varphi $ from $ L^2(\R^n).$ In other words we look for an algebra $ \CA $ which is invariant under the symplectic Fourier transform for which $ W(\CA)L^2(\R^n) = \CS(\R^n).$

For a radial function $ F $ on $ \R^{2n} $, consider the Laguerre expansion
$$ F = \sum_{k=0}^\infty \frac{2^{1-n} k!}{(k+n-1)!} \left( \int_0^\infty F(s) \varphi_k^{n-1}(s) s^{2n-1} ds \right) \varphi_k^{n-1} $$
where $ \varphi_k^{n-1}(s) = L_k^{n-1}(\frac{1}{2}s^2) e^{-\frac{1}{4}s^2}$ are the Laguerre functions of type $ (n-1).$ Now if we make use of the relation $W(\varphi_k^{n-1}) = (2\pi)^n P_k $ (see Theorem 1.3.6 in \cite{Th2}) where $P_k$ are the orthogonal projections associated to the eigenspaces of the Hermite operator $ H $, we see that the Weyl transform $ W(F) $ reduces to the Laguerre transform and we have
$$ W(F) = c_n \sum_{k=0}^\infty \frac{k!}{(k+n-1)!} \left(\int_0^\infty F(s) \varphi_k^{n-1}(s) s^{2n-1} ds \right) P_k .$$
In view of the above formula for radial functions $F$, $ W(F)\varphi $ is Schwartz for all $ \varphi \in L^2(\R^n) $ if and only if $ F $ is Schwartz. In fact, this turns out to be the case even for arbitrary functions.

\begin{thm} The Schwartz space $ \CS(\R^{2n}) $ is an algebra under twisted convolution and $ W(\CS(\R^{2n}))L^2(\R^n) = \CS(\R^n)$ holds.
\end{thm}

\begin{proof} Verifying that the Schwartz space $ \CS(\R^{2n}) $ is an algebra under twisted convolution is a simple exercise. Let us recall Moyal's formula for the Fourier-Wigner transform (\cite{Th2}) which says that for any $\varphi_i \in L^2(\R^n),~ i=1,2,3,4$,
$$\big(V(\varphi_1,\varphi_2), ~V(\varphi_3,\varphi_4) \big)
=(\varphi_1,\varphi_3) \overline{(\varphi_2,\varphi_4)}.$$
An easy calculation using Moyal's formula shows that
$$ W\big(V(\varphi_1,\varphi_2)\big)~\varphi_3 = (\varphi_3,\widetilde{\varphi_2}) \widetilde{\varphi_1}$$
where $\widetilde{\varphi}_i(x) = \varphi_i(-x)$. Now for $\varphi \in \CS(\R^n)$, choose any $f \in L^2(\R^n)$ such that $(f,\varphi) \neq 0$ and then by the above relation
$$ \varphi = \frac{1}{(f,\varphi)}W\big(V(\widetilde{\varphi},
\widetilde{\varphi}) \big)~f. $$
Now
\Bea V(\widetilde{\varphi}, \widetilde{\varphi})(x,u) &=& (2\pi)^{-\frac{n}{2}} \int_{\R^n} e^{i(x \cdot \xi + \frac{1}{2} x \cdot u)} \varphi(-\xi - u)~ \overline{\varphi(-\xi)} d\xi \\
&=& (2\pi)^{-\frac{n}{2}} \int_{\R^n} e^{ix \cdot \xi} \varphi(-\xi - \frac{u}{2})~ \overline{\varphi(-\xi + \frac{u}{2})} d\xi.
\Eea
$\varphi$ being in $\CS(\R^n)$ it follows that $ \varphi(-\xi - \frac{u}{2})~ \overline{\varphi(-\xi + \frac{u}{2})}$ belongs to $\CS(\R^{2n})$ as a function of $(\xi,u).$ Now we make use of the fact that Schwartz space is invariant under partial Fourier transforms to conclude that $V(\widetilde{\varphi}, \widetilde{\varphi}) \in \CS(\R^{2n})$. This proves $\CS(\R^n) \subseteq W(\CS(\R^{2n}))L^2(\R^n).$

On the other hand, take any $F \in \CS(\R^{2n})$ and recall that the Weyl transform can be written as the integral operator
$$ \big( W(F) \varphi \big)(\xi) = (2\pi)^{\frac{n}{2}} \int_{\R^n} (\mathcal{F}_1 F) \big( \frac{-u-\xi}{2}, u-\xi \big)~ \varphi(u)du$$
for $\varphi \in L^2(\R^n)$. Since the Schwartz space is invariant under partial Fourier transforms, it follows easily that the kernel $(\mathcal{F}_1 F) \big( \frac{-u-\xi}{2}, u-\xi \big)$ is in Schwartz class as a function of $(u,\xi).$ Since the partial derivatives of $\big( W(F) \varphi \big)(\xi)$ are also integral operators having finite linear combinations of the partial derivatives of $(\mathcal{F}_1F) \big( \frac{-u-\xi}{2}, u-\xi \big)$ as their kernels, it is enough to check that $\big( W(F) \varphi \big)(\xi)$ with Schwartz class kernel decays faster than the inverse of any polynomial. This follows directly if we apply the Cauchy-Schwarz inequality to the integral.
\end{proof}

%------------------------------------------------------------------------

\section{ A representation theorem for analytic vectors}
\setcounter{equation}{0}

We let $ V $ stand for $ L^2(\R^n) $ so that $ V^\omega $ will denote the space of all analytic vectors for $ \pi.$ If we denote by $ \CM_t^\omega = \CM_t^\omega(\R^n) $ the image of $ V $ under the Poisson-Hermite semigroup $ e^{-t\sqrt{H}} $, then it is well known that $ V^\omega = \cup_{t>0} \CM_t^\omega.$ One way to see this is to make use of Gutzmer's formula for the Hermite expansions proved in \cite{Th1}.

Recall that $ \varphi \in V^\omega $ if and only if the map $(x,u) \rightarrow \pi(x,u)\varphi $ is real analytic from $\R^{2n} $ into $ V $ and hence extends as a holomorphic function to a tube domain in $ \C^{2n} $ such that
$$ \int_{\R^n} |(\pi(x+iy,u+iv)\varphi)(\xi)|^2 d\xi < \infty$$
for all $ (y,v) $ with $ y^2+v^2 < t^2 $ for some $ t> 0.$ If the above is true, then we have the formula (see \cite{Th1})
$$ \int_{\R^n} \int_{K}|\pi(k \cdot (x+iy,u+iv)) \varphi(\xi)|^2 dk d\xi $$
$$ = e^{(u\cdot y-v\cdot x)}\sum_{k=0}^\infty \|P_k\varphi\|_2^2 \frac{k!(n-1)!} {(k+n-1)!} \varphi_k^{n-1}(2iy,2iv).$$
In the above formula, $ K $ stands for the compact group 
$ Sp(n,\R)\cap O(2n,\R), $ and $ P_k $ are the orthogonal projections associated to the eigenspaces of the Hermite operator $ H $ and $ \varphi_k^{n-1}(y,v) = L_k^{n-1}(\frac{1}{2}(y^2+v^2)) e^{-\frac{1}{4}(y^2+v^2)} $ are the Laguerre functions of type $ (n-1).$

The asymptotic behavior of the Laguerre functions $ \varphi_k^{n-1} $ in the complex domain is well known. From Theorem 8.22.3 (Page 199, Szego \cite{Sz}) we see that $\varphi_k^{n-1}(2iy,2iv) $ behaves like $e^{2\sqrt{(2k+n)}(y^2+v^2)^{\frac{1}{2}}}$ and hence our claim is proved. The above Gutzmer's formula also suggests that we look at the space of functions $ \varphi \in V $ for which
$$ \sum_{k=0}^\infty \|P_k\varphi\|_2^2 \frac{k!(n-1)!} {(k+n-1)!} \varphi_k^{n-1}(2iy,2iv) < \infty $$
for all $ y^2+v^2 \leq t^2.$ It turns out that this space, which we denote by $ V_t^\omega, $ is a reproducing kernel Hilbert space.

Indeed, let $ w_t(x,y) $ stand for the weight function
$$ w_t(x,y) = (t^2-y^2)_{+}^{\frac{n}{2}} \frac{J_{\frac{n}{2}-1} (2i|x|(t^2-y^2)^{\frac{1}{2}})} {(2i|x|(t^2-y^2)^ {\frac{1}{2}})^{\frac{n}{2}-1}}$$
and consider the weighted Bergman space of holomorphic functions in the tube domain $ \Omega_t = \{ x+iy \in \C^n: |y| < t \} $ for which
$$ \int_{\Omega_t} |F(x+iy)|^2 w_t(x,y) dx dy < \infty.$$
This is a reproducing kernel Hilbert space and we have

\begin{thm} The above weighted Bergman space coincides with $ V_t^\omega $ with equality of norms. That is to say, $ \varphi \in V_t^\omega$ if and only if $ \varphi $ extends to $ \Omega_t $ as a holomorphic function and
$$ \int_{\Omega_t} |\varphi(x+iy)|^2 w_t(x,y) dx dy $$
$$ = c_n \sum_{k=0}^\infty \|P_k\varphi\|_2^2 \frac{k!(n-1)!} {(k+n-1)!} L_k^{n-1}(-2t^2) e^{t^2}.$$
\end{thm}

This theorem has been proved in \cite{Th3} (Theorem 6.5 in \cite{Th3}) and in view of this we obtain

\begin{thm} For the Schr\"odinger representation $ \pi $ we have
$$ V^\omega = \cup_{t>0} V_t^\omega = \cup_{t>0} \CM_t^\omega.$$
\end{thm}

As $ V_t^\omega $ are reproducing kernel Hilbert spaces, it is possible to characterise $ V^\omega $ in terms of pointwise estimates. Using the asymptotic properties of the Bessel function $J_{n/2-1}(z)$ one can prove the following result.

\begin{thm} A function $ \varphi \in L^2(\R^n) $ is an analytic vector for $ \pi $ if and only if for some $ t > 0 $ it extends to $ \Omega_t $ as a holomorphic function and satisfies the estimate
\bea \label{**} |\varphi(x+iy)| \leq C(t) e^{-|x|(t^2-y^2)^{\frac{1}{2}}},
~~~ \forall (x+iy) \in \Omega_t. \eea
More precisely, every vector $\varphi$ in $V_t^\omega$ satisfies (\ref{**}) for all $s<t$, and conversely any $ \varphi \in L^2(\R^n) $ satisfying (\ref{**}) for some $t>0$ belongs to $ V_s^\omega $ for every $s<t$.
\end{thm}
\begin{proof}
For $n =1$ this was proved by Hille in \cite{Hi}, see also Goodman \cite{G1}. We will prove the theorem only for $n>1.$ Let $ \varphi \in L^2(\R^n) $ be such that for some $ t > 0$ it extends to $ \Omega_t $ as a holomorphic function and satisfies (\ref{**}). Using the Poisson representation of Bessel functions we have
$$ \frac{J_{\frac{n}{2}-1} (2i|x|(t^2-y^2)^{\frac{1}{2}})} {(2i|x|(t^2-y^2)^{\frac{1}{2}})^{\frac{n}{2}-1}} = c_n \int_{-1}^1 e^{-2|x|s(t^2-y^2)^{\frac{1}{2}}} (1-s^2)^{(n-3)/2}ds$$
which holds for every $n>1$. And thus we have
$$ C_{r,n}~ e^{2r|x|(t^2-y^2)^{\frac{1}{2}}} \leq \frac{J_{\frac{n}{2}-1} (2i|x|(t^2-y^2)^{\frac{1}{2}})} {(2i|x|(t^2-y^2)^{\frac{1}{2}})^{\frac{n}{2}-1}} \leq C_n~ 
e^{2|x|(t^2-y^2)^{\frac{1}{2}}}$$ for every $0<r<1.$

The second inequality together with the given condition on $\varphi$ implies
$$ \int_{\Omega_t} |\varphi(x+iy)|^2 w_s(x,y) dx dy < \infty$$
for any $s<t$. Therefore by Theorem 3.1, $\varphi \in V_s^\omega$.

On the other hand, let $\varphi \in V_t^\omega$ be given. Choose and fix $s$ such that $0<s<t$. Also choose $\delta > 0$ small enough so that $s+\delta<t$. Now for any $z \in \Omega_s$, the mean value theorem applied to the balls $B(0,\delta) = \{ w \in \C^{2n} : |w|<\delta \}$ gives
$$ \varphi(z) = \frac{c_n}{\delta^{2n}} \int_{B(0,\delta)} \varphi(z+w)dw.$$
Therefore by Cauchy-Schwarz inequality
\Bea |\varphi(z)|^2 &\leq& \Big(\frac{c_n}{\delta^{2n}} \Big)^2 (\int_{B(0,\delta)} |\varphi(z+w)|^2 w_t(z+w)dw) \\
&& \hspace{0.5in} (\int_{B(0,\delta)} w_t^{-1}(z+w)dw) \\
&\leq& \Big(\frac{c_n}{\delta^{2n}} \Big)^2 ||\varphi||_{V_t^\omega}^2 \int_{B(0,\delta)} w_t^{-1}(z+w)dw. \Eea
Thus we only need to estimate the last integral. But the integrand is dominated by
$$ C_{r,n}^{-1} \big(t^2 - |y+v|^2 \big)^{-\frac{n}{2}}~e^{-2r|x| \big(t^2-|y+v|^2 \big)^{1/2}}.$$
Choosing $r$ such that $rt>s$, one can take $\delta$ small enough to ensure that the above estimate is bounded by a constant (depending on $r$ and $\delta$) multiple of $ e^{-2|x| \big(s^2-|y|^2 \big)^{1/2}}$ for every $z=x+iy \in \Omega_s.$
\end{proof}

We now turn our attention to representation theorems for analytic vectors. In order to state our first result we need to introduce some more notation. For $ \lambda \neq 0 $ we let $ H(\lambda) = -\Delta +\lambda^2 |x|^2 $ stand for the scaled Hermite operator on $ \R^n.$ We use the same notation for the Hermite operator on $ \R^{2n} $ also; the context will make it clear which operator is used. Let $ \Phi_\alpha, \alpha \in \N^n $ denote the Hermite functions on $ \R^n $ which satisfy $ H\Phi_\alpha = (2|\alpha|+n) \Phi_\alpha.$ Defining $ \Phi_\alpha^\lambda(x) = |\lambda|^{n/4}\Phi_\alpha( |\lambda|^{1/2}x) $ the tensor products $ \Phi_\alpha^\lambda \otimes \Phi_\beta^\lambda(x,y) = \Phi_\alpha^\lambda(x)\Phi_\beta^\lambda(y) $ are eigenfunctions of $ H(\lambda) $ on $ \R^{2n} $ with eigenvalues $(2|\alpha|+ 2|\beta|+2n)|\lambda|.$ If we let $\Phi_{\alpha,\beta}^\lambda$ stand for the scaled special Hermite functions, then it is known that $\Phi_{\alpha,\beta}^\lambda$ are eigenfunctions of $H(\lambda/2)$ with eigenvalue $(|\alpha|+ |\beta|+ n)|\lambda|.$ Let us define $\CM_{t,\lambda}^\omega(\R^n) = e^{-t\sqrt{H(\lambda)}} (L^2(\R^n))$ so that $ \CM_t^\omega = \CM_{t,1}^\omega(\R^n).$ With these notations we have

\begin{thm} $ W(\CM_{\sqrt{2}t, 1/2}^\omega(\R^{2n}))V = \CM_t^\omega.$
\end{thm}
\begin{proof}
We will first show that $ \CM_t^\omega \subseteq W(\CM_{\sqrt{2}t,1/2}^ \omega(\R^{2n}))V.$ For this, take $\varphi \in \CM_t^\omega$ and choose any $f \in V$ such that $(f,\varphi) \neq 0$ and then again by Moyal's formula
$$ \varphi = \frac{1}{(f,\varphi)}W\big(V(\widetilde{\varphi}, \widetilde{\varphi}) \big)f.$$
Notice that $\varphi \in \CM_t^\omega$ implies that $\widetilde{\varphi}$ also belongs to $\CM_t^\omega$ and thus we only need to show $V(\widetilde{\varphi}, \widetilde{\varphi}) \in \CM_{\sqrt{2}t,1/2}^\omega(\R^{2n})$. But this follows from the following lemma.

\begin{lem}
For any functions $\varphi_1, \varphi_2 \in \CM_t^\omega$, we have $V(\varphi_1, \varphi_2) \in \CM_{\sqrt{2}t,1/2}^\omega(\R^{2n}).$
\end{lem}
\begin{proof} From Moyal's formula, we have
$$ \big( V(\varphi_1, \varphi_2),~ \Phi_{\alpha,\beta} \big) = (\varphi_1,~\Phi_\alpha) \overline{(\varphi_2,~\Phi_\beta)}.$$
Therefore,
\Bea && \sum_{\alpha,\beta \in \N^n} e^{(2|\alpha|+2|\beta|+2n)^ {\frac{1}{2}}2t} \big| \big( V(\varphi_1, \varphi_2),~ \Phi_{\alpha,\beta} \big) \big|^2 \\
&=& \sum_{\alpha,\beta \in \N^n} e^{(2|\alpha|+2|\beta|+2n)^ {\frac{1}{2}}2t} |(\varphi_1,~\Phi_\alpha)|^2 |{(\varphi_2,~\Phi_\beta)}|^2\\
&\leq& \Big( \sum_{\alpha \in \N^n} e^{(2|\alpha|+n)^ {\frac{1}{2}}2t} |(\varphi_1,~\Phi_\alpha)|^2 \Big) \Big( \sum_{\beta \in \N^n} e^{(2|\beta|+n)^{\frac{1}{2}}2t}|
(\varphi_2,~\Phi_\beta)|^2 \Big) \\
&<& \infty.
\Eea
Hence the lemma.
\end{proof}

Now in order to complete the proof of theorem 3.4 it remains to be shown that $W(\CM_{\sqrt{2}t,1/2}^\omega(\R^{2n}))V \subseteq \CM_t^\omega.$ For this, take any $F \in \CM_{\sqrt{2}t,1/2}^\omega(\R^{2n})$ and $f \in V$. Thus there exists $G \in L^2(\R^{2n})$ such that $F = e^{-t\sqrt{2H(\frac{1}{2})}}G$, which means $$ F = \sum_{\beta,\gamma \in \N^n} e^{-(2|\beta|+2|\gamma|+2n)^ {\frac{1}{2}}t}~ (G,\Phi_{\beta,\gamma}) \Phi_{\beta,\gamma}.$$
Now,
\Bea \big( W(F)f, \Phi_\alpha \big) &=& \sum_{\beta,\gamma \in \N^n} e^{-(2|\beta|+2|\gamma|+2n)^{\frac{1}{2}}t}~
(G,\Phi_{\beta,\gamma}) \Big( W \big( V(\Phi_\beta,\Phi_\gamma) \big) f, \Phi_\alpha \Big) \\
&=& \sum_{\beta,\gamma \in \N^n} e^{-(2|\beta|+2|\gamma|+2n)^ {\frac{1}{2}}t}~ (G,\Phi_{\beta,\gamma}) (f,\widetilde{\Phi_\gamma}) (\widetilde{\Phi_\beta},\Phi_\alpha)\\
&=& (-1)^{|\alpha|} \sum_{\gamma \in \N^n} (-1)^{|\gamma|} e^{-(2|\alpha|+2|\gamma|+2n)^{\frac{1}{2}}t}~ (G,\Phi_{\alpha,\gamma}) (f,\Phi_\gamma). \Eea
Therefore, by Cauchy-Schwarz inequality
\Bea |\big( W(F)f, \Phi_\alpha \big)|^2 &\leq& \big( \sum_{\gamma \in
\N^n}e^{-(2|\alpha|+2|\gamma|+2n)^{\frac{1}{2}}2t}~ |(G,\Phi_{\alpha,\gamma})|^2 \big) \big( \sum_{\gamma \in \N^n} |(f, \Phi_\gamma)|^2 \big) \\
&\leq& e^{-(2|\alpha|+n)^{\frac{1}{2}}2t} ||f||_{L^2(\R^n)}^2 \sum_{\gamma \in \N^n}|(G,\Phi_{\alpha,\gamma})|^2 \Eea
which shows that
$$ \sum_{\alpha \in \N^n} e^{(2|\alpha|+n)^{\frac{1}{2}}2t} |\big( W(F)f, \Phi_\alpha \big)|^2~ \leq~ ||f||_{L^2(\R^n)}^2~ ||G||_{L^2(\R^{2n})}^2 ~<~ 
\infty.$$
Thus $W(F)f \in \CM_t^\omega$ and hence the theorem.
\end{proof}

As we have very good estimates for functions from $\CM_t^\omega(\R^n)$ (in view of Theorem 3.3) we immediately get

\begin{cor} Every analytic vector $ \varphi \in \CM_t^\omega $ has the representation $ \varphi = \pi(h)\psi $ for some $ \psi \in L^2(\R^n) $ and $ h \in L^1(\H^n)$ satisfying
$$ |\CF_1 h^1(x,u)| \leq C e^{-t(x^2+\frac{1}{4}u^2)^ {\frac{1}{2}}},~~~ |\CF_2 h^1(x,u)| \leq C e^{-t(\frac{1}{4}x^2+u^2)^{\frac{1}{2}}}.$$
\end{cor}
\begin{proof} We choose any $f \in V$ such that $(f,\varphi) \neq 0$ and consider
$$ h(x,u,t) = \frac{1}{(f,\varphi)} V(\widetilde{\varphi}, \widetilde{\varphi})(x,u)q(t)$$
where $q \in \CS(\R)$ satisfying $\hat{q}(-1) = 1$ is chosen arbitrarily so that $h^1 = \frac{1}{(f,\varphi)} V(\widetilde{\varphi}, \widetilde{\varphi})$ and therefore $\varphi = \pi(h)f.$ Also,
\Bea
\CF_1 h^1 (x,u) &=& \frac{1}{(f,\varphi)} \CF_1 (V(\widetilde{\varphi}, \widetilde{\varphi}))(x,u)\\
&=& \frac{1}{(f,\varphi)} \widetilde{\varphi}(x + \frac{u}{2}) \overline{\widetilde{\varphi}(x - \frac{u}{2})}.
\Eea
Now we apply Theorem 3.3 to $\widetilde{\varphi}$ with $s = t/ \sqrt{2}$ to get the stated estimate on $\CF_1 h^1$. Similar argument works for $\CF_2 h^1$ and hence the Corollary is proved.
\end{proof}

It is natural to ask if the converse is also true, i.e., if a function $ h $ satisfies the two conditions stated in the above corollary then $ \pi(h) $ takes $ L^2(\R^n) $ into $ \CM_s^\omega$ for some $ s.$ We have the following theorem answering this question in the affirmative.

\begin{thm} Let $F$ be a function on $\R^{2n}$ satisfying the conditions
$$ |\CF_1 F(x,u)| \leq C e^{-t(x^2+\frac{1}{4}u^2)^ {\frac{1}{2}}},~~~ |\CF_2 F(x,u)| \leq C e^{-t(\frac{1}{4}x^2+u^2)^{\frac{1}{2}}}.$$ Then for any $ \varphi \in L^2(\R^n) $ and $0 < s < \frac{t}{2\sqrt{2n}},~ W(F)\varphi $ belongs to $\CM_s^\omega.$ Moreover, $ W(F): L^2(\R^n) \rightarrow \CM_s^\omega $ is bounded.
\end{thm}

In order to prove the theorem we require the following lemma.

\begin{lem} Under the hypothesis of Theorem 3.7 the function $ \psi = W(F) \varphi,~ \varphi \in L^2(\R^n) $ satisfies the estimates
$$ |\psi(\xi)| \leq C_{n,t} ||\varphi||_2~ e^{-\frac{t}{2}|\xi|},~~~ |\hat{\psi}(\eta)| \leq C_{n,t} ||\varphi||_2~ e^{-\frac{t}{2}|\xi|}.$$
\end{lem}
\begin{proof}
Recall that the Weyl transform can be written as an integral operator and
$$\psi(\xi) = (2\pi)^{\frac{n}{2}} \int_{\R^n} (\CF_1 F) \big( \frac{-u-\xi}{2}, u-\xi \big)~ \varphi(u)~du.$$
Using Parseval's identity for the Fourier transform, we get
$$\hat{\psi}(\eta) = (2\pi)^ {\frac{n}{2}}\int_{\R^n} (\CF_2 F)\big(u+\eta, \frac{u-\eta}{2} \big)~ \widehat{\varphi}(-u)~du.$$
Now, using the given hypothesis on $\CF_1 F$, we get
\Bea |\psi(\xi)| &\leq& C_n \int_{\R^n} e^{-t \big(\frac{u^2+\xi^2+2u.\xi}{4} + \frac{u^2+\xi^2-2u.\xi}{4} \big)^{\frac{1}{2}}} |\varphi(u)| du \\
&=& C_n \int_{\R^n} e^{-t \big(\frac{u^2+\xi^2}{2} \big)^{\frac{1}{2}}} |\varphi(u)| du \\
&\leq& C_n \int_{\R^n} e^{-\frac{t}{2} (|u|+|\xi|)} |\varphi(u)| du \\
&\leq& C_{n,t} ||\varphi||_2~e^{-\frac{t}{2}|\xi|}. \Eea Similarly we get the estimate for $\hat{\psi}(\eta)$.
\end{proof}

Once we have proved the above estimates, we can appeal to the following result to conclude the proof of Theorem 3.7.

\begin{thm} Let $ \psi \in L^2(\R^n) $ satisfy
$$ |\psi(\xi)| \leq C e^{-t|\xi|},~~~ |\hat{\psi}(\eta)| \leq C e^{-t|\eta|}.$$ Then $$ |(\psi,\Phi_\alpha)| \leq C_{n,t} \prod_{j=1}^n (2\alpha_j+1)^{1/4} e^{-\frac{t}{\sqrt{2n}} (2|\alpha|+n)^{\frac{1}{2}}}$$ for all $\alpha \in \N^n.$
\end{thm}

\begin{proof} In order to prove the above theorem we adopt the method used by Vemuri in \cite{V}. We make use of the Bargmann transform $ B $ which takes $ L^2(\R^n)$ isometrically onto the Fock space consisting of all entire functions on $ \C^n $ that are square integrable with respect to the Gaussian measure $ (4\pi)^{-n/2}e^{-1/2 |z|^2} dz. $ The transform $ B $ is explicitly given by
$$ Bg(z)= \pi^{-\frac{n}{2}}e^{- \frac {1}{4}z^2} \int_{\mathbb{R}^n}g(\xi)~ e^{- \frac{1}{2}\xi^2} e^{z \cdot \xi}
d\xi $$
where $g \in L^2(\R^n)$ and $z \in \C^n .$ The most important property of $B$ which we need is that the Taylor coefficients $c_\alpha$ of $Bg$ are related to the Hermite coefficients $(g,\Phi_\alpha)$ of $g.$ More precisely, we have
$$ (g,\Phi_\alpha) = \big( 2^\alpha \alpha!~\pi^{n/2} \big)^{\frac{1}{2}} c_\alpha.$$
Therefore, in order to prove the theorem we only need to estimate the Taylor coefficients of $ B\psi $ for which, in view of Cauchy's formula, we need good estimates of $ B\psi.$ In estimating $ B\psi $ we make use of another important property of the Bargmann transform, viz., $ Bg(-iz) = B\hat{g}(z).$
The given condition on $\psi$ implies that
\Bea
|B\psi(x+iy)| &\leq& C \pi^{-\frac{n}{2}} e^{-\frac{1}{4}(x^2-y^2)} \int_{\R^n} e^{-t|\xi|} e^{- \frac{1}{2}\xi^2} e^{|x||\xi|} d\xi \\
&\leq& C_n~ e^{-\frac{1}{4}(x^2-y^2)} \int_0^\infty e^{-\frac{1}{2}r^2} e^{(|x|-t)r} dr \\
&\leq& C_n~ e^{-\frac{1}{4}(x^2-y^2)} e^{\frac{1}{2}(|x|-t)^2} \\
&=& C_{n,t}~ e^{\frac{1}{4}(x^2+y^2)} e^{-t|x|} \\
&\leq& C_{n,t}~ \prod_{j=1}^n e^{\frac{1}{4}(x_j^2+y_j^2)} e^{-\frac{t}{\sqrt{n}}|x_j|}.
\Eea
Similarly, the given condition on $\hat{\psi}$ and the relation $B\psi(-iz) = B\hat{\psi}(z)$ gives the other estimate for $B\psi$, namely,
$$ |B\psi(x+iy)| \leq C_{n,t}~ \prod_{j=1}^n e^{\frac{1}{4}(x_j^2+y_j^2)} e^{-\frac{t}{\sqrt{n}}|y_j|}.$$
Using the Cauchy integral formula, we get for every $r_j > 0,~ j=1,2,..,n$
\Bea
|c_\alpha| &\leq& \left(\frac{1}{2\pi} \right)^n \int_0^{2\pi} ... \int_0^{2\pi} \frac {|B\psi(r_1 e^{i\theta_1},...,r_n e^{i\theta_n})|} {r_1^{\alpha_1}...r_n^{\alpha_n}} d\theta_1...d\theta_n \\
&\leq& 4 \left(\frac{1}{2\pi} \right)^n C_{n,t} \prod_{j=1}^{n} r_j^{-\alpha_j}~ e^{\frac{1}{4}r_j^2} \left( \int_0^{\frac{\pi}{4}} e^{-\frac{t}{\sqrt{n}} r_j \cos \theta_j} d\theta_j + \int_{\frac{\pi}{4}}^{\frac{\pi}{2}} e^{-\frac{t}{\sqrt{n}} r_j \sin \theta_j} d\theta_j \right) \\
&\leq& \tilde{C}_{n,t} \prod_{j=1}^{n} r_j^{-\alpha_j}~ e^{\frac{1}{4}r_j^2} e^{-\frac{t}{\sqrt{2n}}r_j}.
\Eea
Since the above is true for every $r_j > 0,~ j=1,2,..,n$, we can take in particular $r_j = (2\alpha_j+1)^{\frac{1}{2}}$ to get
\Bea |c_\alpha| &\leq& \tilde{C}_{n,t} \prod_{j=1}^n (2\alpha_j+1)^{-\alpha_j/2}~
e^{\frac{1}{4}(2\alpha_j+1)} e^{-\frac{t}{\sqrt{2n}} (2\alpha_j+1)^{\frac{1}{2}}}.
\Eea
And thus
\Bea |(\psi,\Phi_\alpha)| &\leq& \tilde{C}_{n,t} \big( 2^\alpha \alpha! \pi^{n/2} \big)^{\frac{1}{2}} \prod_{j=1}^n (2\alpha_j+1)^{-\alpha_j/2}~ e^{\frac{1}{4}(2\alpha_j+1)} e^{-\frac{t}{\sqrt{2n}}(2\alpha_j+1)^{\frac{1}{2}}} \\
&\sim& \tilde{C}_{n,t} \prod_{j=1}^n (2\alpha_j+1)^{1/4}~ e^{-\frac{t}{\sqrt{2n}}(2\alpha_j+1)^{\frac{1}{2}}} \\
&\leq& \tilde{C}_{n,t} \prod_{j=1}^n (2\alpha_j+1)^{1/4} e^{-\frac{t}{\sqrt{2n}}(2|\alpha|+n)^{\frac{1}{2}}}
\Eea
where the second last estimate is obtained using the Stirling's formula $\Gamma(\lambda+1) \sim \lambda^{\lambda+\frac{1}{2}}~ e^{-\lambda}.$
\end{proof}

Proof of Theorem 1.1: The fact that $\pi(\CB^1)V = V^\omega$ is already proved in Corollary 3.6 and Theorem 3.7. All we need to check is that $ \CB^1 $ is an algebra. For this, take any $g,h \in \CB^1$, and choose $t>0$ such that $g^1_1,~ g^1_2,~ h^1_1$ and $h^1_2$ satisfy the decay condition for the same $t.$ Now as easy application of change of variables implies that
$$ \CF_1 (g^1 \times h^1)(x,u) = c_n \int_{\R^n} (\CF_1 g^1) (x+\frac{v}{2},~ u-v) (\CF_1 h^1)(x+\frac{v-u}{2},~ v) dv.$$
The integrand in the above integral is bounded by
\Bea
&& e^{-t \big(x^2 + \frac{1}{2}v^2 + \frac{1}{4}u^2 + x \cdot v -\frac{1}{2} u \cdot v \big)^{\frac{1}{2}}}~ e^{-t \big(x^2+\frac{1}{2}v^2 + \frac{1}{4}u^2 + x \cdot v - x \cdot u -\frac{1}{2} u \cdot v \big)^{\frac{1}{2}}} \\
&\leq& e^{-t \big(2x^2 + \frac{1}{2}u^2 + v^2 + 2x \cdot v - u \cdot v - x \cdot u \big)^{\frac{1}{2}}} \\
&=& e^{-\sqrt{2}t \big(x^2 + \frac{1}{4}u^2 + \frac{1}{2}v^2 + v \cdot (x - \frac{1}{2}u) - \frac{1}{2}x \cdot u \big)^{\frac{1}{2}}} \\
&=& e^{-\sqrt{2}t \big(\frac{1}{2}(x^2 + \frac{1}{4}u^2) + \frac{1}{2}(v +(x - \frac{1}{2}u))^2 \big)^{\frac{1}{2}}} \\
&\leq& e^{-\frac{t}{\sqrt{2}} \big((x^2 + \frac{1}{4}u^2)^{\frac{1}{2}} + |v +(x - \frac{1}{2}u)| \big)}.
\Eea
And thus,
\Bea |\CF_1(g \ast h)^1(x,u)| &=& |\CF_1 (g^1 \times h^1)(x,u)| \\
&\leq& C_{n,t} e^{-\frac{t}{\sqrt{2}} (x^2 + \frac{1}{4}u^2)^{\frac{1}{2}}}.
\Eea
Similar argument can be used to estimate $\CF_2(g \ast h)^1(x,u)$ also. Hence $g \ast h \in \CB^1$, proving that $\CB^1$ is an algebra.

%------------------------------------------------------------------------

\section{ Hardy's theorem and entire vectors}
\setcounter{equation}{0}

In this section we prove a representation theorem for certain classes of entire vectors of the Schr\"odinger representation $ \pi.$ Recall that a function $ \varphi \in L^2(\R^n) $ is called entire if it belongs to $ \cap_{t>0} V_t^\omega = \cap_{t>0}\CM_t^\omega.$ Clearly, entire vectors extend to the whole of $ \C^n $ as entire functions and they satisfy the estimates stated in Theorem 3.3 for all $ t > 0.$ We denote the space of entire vectors by $ E.$ For each $ t > 0 $ let $ E_t $ stand for the image of $ L^2(\R^n) $ under the Hermite semigroup $ e^{-tH}.$ Then $ E_t \subset E $ and the holomorphically extended functions are square integrable with respect to the weight function $$ U_t(x,y) = 2^n (\sinh(4t))^{-\frac{n}{2}} e^{\tanh(2t)x^2 - \coth(2t)y^2}.$$
In other words, $ E_t $ consists precisely of functions on $\R^n$ which are restrictions of functions from the Hermite-Bergman space which is denoted by $ \CH_t(\C^n).$ We also know that for every $ \varphi \in L^2(\R^n) $
$$ \int_{\R^{2n}} |e^{-tH}\varphi(x+iy)|^2 U_t(x,y) dx dy = \int_{\R^n} |\varphi(x)|^2 dx.$$
For more about these spaces we refer to \cite{By} and \cite{Th1}.

\begin{thm} $ E_H = \cup_{t >0 } E_t $ and $ E_\infty = \cap_{t>0} E_t$ are algebras under convolution as well as under pointwise multiplication.
\end{thm}

To see this, note that both $ E_H $ and $ E_\infty $ are invariant under the Fourier transform. And since the Fourier transform intertwines convolution and pointwise multiplication, it is sufficient to check if $ FG \in E_H $ (resp.$~ E_\infty $) whenever $ F, G \in E_H $ (resp.$~ E_\infty $). But this follows easily once we know the following result (Theorem 4.2 in \cite{RT}) which characterises the image of $ \CS(\R^n)$ under the Hermite semigroup $ e^{-tH}.$ More precisely,

\begin{thm}[Radha-Thangavelu]
Let $ t> 0 $ be fixed. Suppose $F$ is a holomorphic function 
on $\C^n.$ Then there exists a function $f\in \CS(\R^n)$ such that $F=e^{-tH}f$ if and only if $F$ satisfies
$$ |F(z)|^2\leq A_m \frac{e^{-(\tanh 2t) x^2+ (\coth 2t)y^2}} {(1+x^2+y^2)^{2m}} $$
for some constants $A_m,~ \forall m = 1,2,3,\cdots.$
\end{thm}

The following analogue of Theorem 3.4 can be easily proved for the spaces $ E_t.$ As the proof is exactly similar to that of Theorem 3.4 we just state the theorem without proof.

\begin{thm} $ W(e^{-2tH(1/2)}L^2(\R^{2n}))V = E_t.$
\end{thm}

In the above representation theorem it is preferable to replace the space $e^{-2tH(1/2)}L^2(\R^{2n})$ by an algebra defined in terms of pointwise estimates. We begin with the observation that if $ \psi = e^{-tH}\varphi $ is from $ E_t $ then $ \psi = W(p_t)\varphi $ where
$$ p_t(x,u) = c_n (\sinh t)^{-n} e^{-\frac{1}{4}(\coth t)(x^2+u^2)} $$
is the heat kernel associated to the special Hermite operator. The above follows from the fact that $ W(p_t) = e^{-tH} $ (see \cite{Th2}). Now as we did earlier, if we define $ h(x,u,s) = p_t(x,u)q(s) $ where $q \in \CS(\R)$ is such that $\hat{q}(-1) = 1,$ so that $h^1 = p_t$, then $h$ satisfies the estimates
$$ |\CF_1 h^1(x,y)| \leq C_{n,t} e^{-\tanh(t)x^2}p_t(0,y),$$
$$ |\CF_2 h^1(x,y)| \leq C_{n,t} e^{-\tanh(t)y^2}p_t(x,0).$$
Therefore, in analogy with Corollary 3.6 we have

\begin{cor} Every element of $ E_t $ is of the form $ \varphi = \pi(h)\psi $ for some $ \psi \in L^2(\R^n) $ and $ h \in L^1(\H^n)$ satisfying the conditions
$$ |\CF_1 h^1(x,u)| \leq C e^{-\tanh(t)(x^2+\frac{1}{4}u^2)},~~~ |\CF_2 h^1(x,u)| \leq C e^{-\tanh(t)(\frac{1}{4}x^2+u^2)}.$$
\end{cor}

It is again natural to ask if the converse is also true, i.e., if a function $ h $ satisfies the two conditions stated in the above corollary then $ \pi(h) $ takes $ L^2(\R^n) $ into $ \CH_s$ for some $ s.$ We have the following theorem answering this question in the affirmative.

\begin{thm} Suppose that $ F $ is a function on $\R^{2n}$ satisfying the conditions
$$ |\CF_1 F(x,u)| \leq C e^{-\tanh(t)(x^2+\frac{1}{4}u^2)},~~~ |\CF_2 F(x,u)| \leq C e^{-\tanh(t)(\frac{1}{4}x^2+u^2)}$$
for some $ t > 0.$ Then for any $ \varphi \in L^2(\R^n) $ and $ 0 < s < t/2n, W(F)\varphi $ belongs to $ E_s.$ Moreover, $ W(F): L^2(\R^n) \rightarrow \CH_s(\C^n) $ is bounded.
\end{thm}

Before proving the above theorem, let us remark that once we have this 
theorem, the arguments given at the end of Section 3 in proving Theorem 1.1 can be modified to prove Theorem 1.2. So we do not repeat the arguments again.

In order to prove Theorem 4.5 we require the following lemma which gives pointwise estimates on $ W(F)\varphi $ when $ F $ is as in the theorem.

\begin{lem} Under the hypothesis of the theorem the function $ \psi = W(F)\varphi $ satisfies the estimates
$$ |\psi(\xi)|\leq C_{n,t} ||\varphi||_2 e^{-\frac{1}{2} \tanh (t)\xi^2},~~~ |\CF\psi(\eta)| \leq C_{n,t} ||\varphi||_2 e^{-\frac{1}{2} \tanh (t)\eta^2}.$$
\end{lem}

The proof of this lemma is similar to that of Lemma 3.8 and hence we skip 
the proof. We now appeal to the following theorem to complete the proof of 
Theorem 4.5.

\begin{thm} Suppose $ \psi \in L^2(\R^n) $ satisfies the estimates
$$ |\psi(\xi)|\leq C e^{-\frac{1}{2}(\tanh 2t)\xi^2},~~~ |\CF\psi(\eta)| \leq C e^{-\frac{1}{2}(\tanh 2t)\eta^2} $$
for some $ t > 0.$ Then the Fourier-Hermite coefficients of $\psi$ satisfy
$$ |(\psi,\Phi_\alpha)| \leq C_{n,t} \Pi_{j=1}^n (2\alpha_j+1)^{-1/4n}
e^{-(2|\alpha|+n)t/2n} $$
for every $ \alpha \in \N^n.$
\end{thm}
\begin{proof} We deduce the result from the one dimensional case which is due to Vemuri \cite{V} (see also \cite{RaT}). Assuming the one dimensional result consider
$$ (\psi,\Phi_\alpha) = \int_{\R^{n}} \psi(x,y)\Phi_\alpha(x,y) dx dy = (\psi_\mu, h_k) $$
where $ \alpha = (\mu,k) \in \N^{n-1} \times \N,~ (x,y) \in \R^{n-1} \times \R$ and
$$ \psi_\mu(y) = \int_{\R^{n-1}} \psi(x,y)\Phi_\mu(x) dx.$$
Using Parseval's theorem for the Fourier transform we can also write
\Bea \psi_\mu(y) &=& \int_{\R^{n-1}} (\CF_x \psi)(u,y)~ \overline{\hat{\Phi}_\mu(u)} du \\
&=& i^{|\mu|} \int_{\R^{n-1}} (\CF_x \psi)(u,y)~ \Phi_\mu(u) du
\Eea
where $\CF_x$ is the Fourier transform in the $u$-variable. Thus
$$ \hat{\psi_\mu}(y) = i^{|\mu|} \int_{\R^n} \hat{\psi}(u,y)\Phi_\mu(u) du.$$
It is now clear from the hypothesis that the one dimensional function $\psi_\mu $ satisfies the estimates
$$ |\psi_\mu(y)| \leq C_{n,t} e^{-\frac{1}{2}\tanh (2t)y^2},~~ |\hat{\psi_\mu}(y)| \leq C_{n,t} e^{-\frac{1}{2} \tanh (2t)y^2}.$$
Hence by the one dimensional result we have
$$ |(\psi,\Phi_\alpha)| = |(\psi_\mu, h_k)| \leq C_{n,t} (2\alpha_n+1)^{-1/4} e^{-(2\alpha_n+1)t/2}.$$
Similar arguments imply that
$$ |(\psi,\Phi_\alpha)| \leq C_{n,t} (2\alpha_j+1)^{-1/4} e^{-(2\alpha_j+1)t/2} $$
for all $ j = 1,2,...,n.$ By combining these estimates we obtain the result.
\end{proof}

\begin{rem} The above theorem deals with Case (iii) of Hardy's uncertainty principle \cite{H}. To be more precise, if a function $ f $ on $ \R^n $ satisfies
$$ |f(x)| \leq C e^{-\frac{1}{2}a x^2 },~~~ |\hat{f}(y)| \leq C e^{-\frac{1}{2}b y^2} $$
then (i) $ f = 0 $ when $ ab > 1;$ (ii) $ f(x) = C e^{-\frac{1}{2}a x^2} $ when $ ab = 1 $ and (iii) there are infinitely many linearly independent functions satisfying the above when $ ab < 1.$ But one can say more about such functions.

It is conjectured that under the hypothesis of Theorem 4.7 (which corresponds to $ a = b = \tanh 2t < 1 $ in Hardy's theorem) the Fourier-Hermite coefficients of $ \psi $ decay like $e^{-(2|\alpha|+n)t/2}.$
In \cite{RaT} we have proved this result under the extra assumption that the spherical harmonic expansion of $ f $ is finite. In the general case the conjecture is still open.
\end{rem}

Proof of Theorem 1.3: As mentioned in the introduction, the result was proved in \cite{Pf} in one dimension. First of all note that case (i) follows immediately from case (ii). By a suitable dilation we can reduce everything to the case $a = b = 1.$

Pfannschmidt first proved that if a function $ f $ on $ \R $ satisfies the 
estimate $ |f(x)| \leq P(x)e^{-\frac{1}{2}x^2}$ where 
$ \limsup_{|x|\rightarrow \infty} \frac{\log{P(x)}}{|x|^\mu}~ =~ 0~$ for 
some $ 0 \leq \mu \leq 2 $ then $ \hat{f} $ extends to $ \C $ as an entire function (see Proposition 3.2 in \cite{Pf}). But the argument given by Pfannschmidt to prove this assertion goes through for functions on $ \R^n,~ n>1$ as well. Hence we know, under the hypotheses of the theorem, that both $ f $ and $\hat{f} $ are entire. Therefore, we can assume, without loss of generality, that both $ P $ and $ Q $ are continuous functions.

Now, the hypothesis together with the continuity of $P$ and $Q$ is equivalent to the condition that for every $t>0$, there exists a constant $C_t$ such that
$$ |f(x)| \leq C_t~ e^{-\frac{1}{2}\tanh(2t)|x|^2},~~~~ |\hat{f}(\xi)| \leq C_t~ e^{-\frac{1}{2}\tanh(2t)|\xi|^2}.$$
So, we can appeal to Theorem 4.7 to conclude that for every $t>0$, there exists a constant $C(t)$ such that for every $\alpha \in \N^n$,
$$ |(f, \Phi_\alpha)| \leq C(t) ~e^{-(2|\alpha|+n) \frac{t}{2n}}.$$
Since $t>0 $ is arbitrary we get the estimate
$$|(f , \Phi_\alpha)| \leq C_1(s)e^{-(2|\alpha|+n)s},$$
for all $ s>0.$ Now by Mehler's formula (see \cite{Th2}) it can be easily shown that $ f $ satisfies the estimate
$$ |f(x+iy)| \leq ~ C_2(s) e^{-\frac{1}{2}\tanh(s)|x|^2 + \frac{1}{2}\coth(s)|y|^2}.$$
As $ \Phi_\alpha $ are eigenfunctions of the Fourier transform with eigenvalues $(-i)^{|\alpha|} $, it follows that $|(\hat{f}, \Phi_\alpha)| = |(f, \Phi_\alpha)| \leq C_1(s) e^{-(2|\alpha|+n)s}$ and hence $ \hat{f} $ also satisfies the same estimate as $ f.$ Since this is true for all $s>0$, we see that $e^{\frac{1}{2}z^2}f(z)$ and $e^{\frac{1}{2}z^2}\hat{f}(z)$ are entire functions of order 2 and at most of minimal type.

\begin{center}
{\bf Acknowledgments}

\end{center}
The work of the first author is supported by Senior Research Fellowship from the Council of Scientific and Industrial Research, India. The work of the second author is supported by J. C. Bose National Fellowship from the Department of Science and Technology (DST), India.

\end{document}